%% file: ms.tex
\documentclass[
		a4paper,	%paper size
		twoside,	%double page
		article, 	%may be relaced by "book"
		english,
		10pt,
		leqno,
	]{memoir}

\newif\ifdanish
\newif\ifbourbakiintervals
\newif\ifshowlabels
\newif\ifindex
\newif\iffront
\newif\ifprenumber % if theorems should use format number. type
\prenumbertrue

\makeatletter

\input{_preamble/lang_enc_font/_lang_enc_font_1}

\usepackage{
	graphicx, 	%importing graphics
	mathtools, 	%extension of amsmath, enabling commands like \abs, \norm
	etoolbox,	%programming facilities,
	textcomp,	%latex symbols; part of standard latex
	microtype, 	%microtypography
}

\usepackage{
	xspace,
	xparse,
	}

\usepackage[shortlabels]{enumitem}%control lists

\input{_preamble/boldmath/_boldmath_1}

\makeatletter

\input{_preamble/csquotes/_csqoutes_1}

\input{_preamble/tikz/_tikz_1}

\input{_preamble/biblatex/_biblatex_1}

	\input{_preamble/ntheorem/_ntheorem_2}

	\input{_preamble/hyperref/_hyperref_1}

	\input{_preamble/cleveref/_cleveref_1}
	
	\input{_preamble/theoremlists/_theoremlists_1}

\title{Homotopy (co)limits via homotopy (co)ends in general combinatorial model categories}
\date{\today}
\author{Sergey Arkhipov\textsuperscript{a}
and
Sebastian Ørsted\textsuperscript{b}
\\
\tiny\textsuperscript{b}Matematisk Institut, Aarhus Universitet,
\\[-.5em]
\tiny Ny Munkegade~118, 8000 Aarhus~C, Denmark
\\[-.5em]
\tiny\href{mailto:hippie@math.au.dk}{hippie@math.au.dk}
\\
\tiny\textsuperscript{b}Matematisk Institut, Aarhus Universitet,
\\[-.5em]
\tiny Ny Munkegade~118, 8000 Aarhus~C, Denmark
\\[-.5em]
\tiny\href{mailto:sorsted@gmail.com}{sorsted@gmail.com}}

\hypersetup{
		pdfauthor={Sergey Arkhipov and Sebastian Ørsted},
		pdftitle={Homotopy (co)limits via homotopy (co)ends in general combinatorial model categories},
	}

\input{_commands/_commands_1}

\input{_doc_commands/_doc_commands_1}

\makeatother

\begin{document}

\maketitle

\input{functors/functors_5016}

\printbibliography

\end{document}

%% file: _preamble/lang_enc_font/_lang_enc_font_1.tex
\ifluatex
	\usepackage{
		fontspec, 		%for opentype fonts
		babel,			% in the future, relace by 
		lualatex-math,	%fixes for math in latex
		}
\else
	\ifxetex
		\usepackage{fontspec}
		\usepackage[babelshorthands=true]{polyglossia}
		}
		\setmainlanguage{danish}
		\setotherlanguage{english}
	\else
		\usepackage[utf8]{inputenc} %for using Unicode
		\usepackage[T1]{fontenc} 	%European fonts
		\usepackage{babel}
		\usepackage[
			noDcommand,
			slantedGreeks,
			]{kpfonts}
	\fi
\fi

%% file: _preamble/boldmath/_boldmath_1.tex
\g@addto@macro\bfseries{\boldmath} %make math in bold text automatically bold

%% file: _preamble/csquotes/_csqoutes_1.tex
\usepackage[autostyle,german=guillemets,english=american]{csquotes}
\DeclareQuoteStyle[myguillemotsstyle]{danish}
{\guillemotright}{\guillemotleft}
{\textquoteleft}{\textquoteright} %secondary quotes: ‘...’

\ExecuteQuoteOptions{danish=myguillemotsstyle}

\MakeAutoQuote{“}{”}

\MakeAutoQuote{‘}{’}

\MakeAutoQuote{»}{«}

%% file: _preamble/tikz/_tikz_1.tex
\usepackage{tikz}

\usetikzlibrary{babel}

\usetikzlibrary{cd}

\tikzcdset{
  arrow style=tikz,
}

\tikzset{
  >/.tip={Stealth[length=2.9pt, width=4.4pt, inset=1.8pt]},
  tikzcd left hook/.tip={Hooks[
	  left,
	  length=2pt,
	  width=5.5pt,
	]},
  iso/.style={
    every to/.append style={
      edge node={
        node [sloped, allow upside down]{
          \raisebox{0.1em}[0pt][0pt]{\ensuremath{\sim}}
        }
      }
    }
  },
  iso'/.style={
    every to/.append style={
      edge node={
        node [sloped, allow upside down]{
          \raisebox{-0.6em}[0pt][0pt]{\ensuremath{\sim}}
        }
      }
    }
  },
  symbol/.style={
      draw=none,
      every to/.append style={
        edge node={node [sloped, allow upside down, auto=false]{$#1$}}}
  },
}

\newcommand\pushout{\arrow[ul,phantom,"\lrcorner",very near start]}
\newcommand\pullback{\arrow[dr,phantom,"\ulcorner",very near start]}

\newcommand\invdot{\mathrlap{.}}

%% file: _preamble/biblatex/_biblatex_1.tex
\usepackage[backend=biber,citestyle=authoryear-ibid,bibstyle=authoryear,
	uniquename=init,giveninits=true,autolang=hyphen]{biblatex}
\nobibintoc %% Removes bibliography from the pable of contents.

%% file: _preamble/ntheorem/_ntheorem_2.tex
\usepackage[thmmarks,amsmath]{ntheorem}

\input{_preamble/ntheorem/ntheorem_config/_ntheorem_config_1}

\input{_preamble/ntheorem/da/_ntheorem_da_1.tex}

\input{_preamble/ntheorem/en/_ntheorem_en_1.tex}

\ifdanish
	\danishtheorems
\else
	\englishtheorems
\fi

\theoremseparator{.}

\theorembodyfont{\normalfont}

\ifprenumber
	\theoremstyle{notitle}
\else
	\theoremstyle{postnumpara}
\fi

\newtheorem{para}[equation]{}

\ifprenumber
	\theoremstyle{prenumber}
\else
	\theoremstyle{customplain}
\fi

\theorembodyfont{\itshape}

\newtheorem{theorem}[equation]{\@nthm@theorem}
\newtheorem{corollary}[equation]{\@nthm@corollary}
\newtheorem{proposition}[equation]{\@nthm@proposition}
\newtheorem{lemma}[equation]{\@nthm@lemma}

\theorembodyfont{\normalfont}

\theoremsymbol{\ensuremath{\triangle}}
\newtheorem{remark}[equation]{\@nthm@remark}

\theoremsymbol{\ensuremath{\scriptstyle\bigcirc}}
\newtheorem{example}[equation]{\@nthm@example}

\theoremstyle{notitle}

\theoremstyle{nonumberplainflex}
\theoremsymbol{\ensuremath{\square}}
\theoremheaderfont{\itshape\bfseries}
\newtheorem{proof}{\@nthm@proof}

\theoremstyle{nonumberplainflex}
\theoremheaderfont{\normalfont\bfseries}
\theoremsymbol{}

\theorembodyfont{\normalfont}

\theorembodyfont{\itshape}

\theorembodyfont{\normalfont}

\theoremsymbol{\ensuremath{\triangle}}

\theoremsymbol{\ensuremath{\scriptstyle\bigcirc}}

\theoremsymbol{}
\theoremheaderfont{\normalfont\bfseries}

\ifprenumber
	\theoremstyle{notitlebreak}
\else
	\theoremstyle{postnumparabreak}
\fi

\ifprenumber
	\theoremstyle{prenumberbreak}
\else
	\theoremstyle{custombreak}
\fi

\theorembodyfont{\itshape}
\newtheorem{theorembreak}[equation]{\@nthm@theorem}

\newtheorem{propositionbreak}[equation]{\@nthm@proposition}
\newtheorem{lemmabreak}[equation]{\@nthm@lemma}

\theorembodyfont{\normalfont}

\theoremsymbol{\ensuremath{\triangle}}

\theoremsymbol{\ensuremath{\scriptstyle\bigcirc}}
\newtheorem{examplebreak}[equation]{\@nthm@example}

\theoremsymbol{}

\theoremstyle{notitlebreak}

\theoremstyle{nonumberbreakflex}
\theoremsymbol{\ensuremath{\square}}
\theoremheaderfont{\itshape\bfseries}

\theoremstyle{nonumberbreakflex}
\theoremheaderfont{\normalfont\bfseries}
\theoremsymbol{}
\theorembodyfont{\normalfont}

\theorembodyfont{\itshape}

\theorembodyfont{\normalfont}

\theoremsymbol{\ensuremath{\triangle}}

\theoremsymbol{\ensuremath{\scriptstyle\bigcirc}}

%% file: _preamble/ntheorem/ntheorem_config/_ntheorem_config_1.tex
\makeatletter
\newtheoremstyle{notitle}% for exercises
{\item[\hskip\labelsep \theorem@headerfont ##2\theorem@separator]}%
{\item[\hskip\labelsep \theorem@headerfont ##2\theorem@separator\ ##3\theorem@separator]}

\newtheoremstyle{notitlebreak}% for exercises with a break
{\item[\rlap{\vbox{\hbox{\hskip\labelsep \theorem@headerfont ##2\theorem@separator}\hbox{\strut}}}]}%
{\item[\rlap{\vbox{\hbox{\hskip\labelsep \theorem@headerfont ##2\theorem@separator\ ##3\theorem@separator}\hbox{\strut}}}]}

\newtheoremstyle{prenumber}%
  {\item[\hskip\labelsep \theorem@headerfont ##2\theorem@separator\ ##1\theorem@separator]}%
  {\item[\hskip\labelsep \theorem@headerfont ##2\theorem@separator\ ##3\theorem@separator]}

\newtheoremstyle{prenumberbreak}%
  {\item[\rlap{\vbox{\hbox{\hskip\labelsep \theorem@headerfont 
          ##2\theorem@separator\ ##1\theorem@separator}\hbox{\strut}}}]}%
  {\item[\rlap{\vbox{\hbox{\hskip\labelsep \theorem@headerfont 
          ##2\theorem@separator\ ##3\theorem@separator}\hbox{\strut}}}]}

\newtheoremstyle{customplain}%
  {\item[\hskip\labelsep \theorem@headerfont ##1\ ##2\theorem@separator]}%
  {\item[\hskip\labelsep \theorem@headerfont ##3\ ##2\theorem@separator]}
  
\newtheoremstyle{custombreak}%
  {\item[\rlap{\vbox{\hbox{\hskip\labelsep \theorem@headerfont
          ##1\ ##2\theorem@separator}\hbox{\strut}}}]}%
  {\item[\rlap{\vbox{\hbox{\hskip\labelsep \theorem@headerfont 
          ##3\ ##2\theorem@separator}\hbox{\strut}}}]}

\newtheoremstyle{postnumpara}%
  {\item[\hskip\labelsep \theorem@headerfont ##2\theorem@separator]}%
  {\item[\hskip\labelsep \theorem@headerfont ##3\ ##2\theorem@separator]}
  
\newtheoremstyle{postnumparabreak}%
  {\item[\rlap{\vbox{\hbox{\hskip\labelsep \theorem@headerfont
          ##2\theorem@separator}\hbox{\strut}}}]}%
  {\item[\rlap{\vbox{\hbox{\hskip\labelsep \theorem@headerfont 
          ##3\ ##2\theorem@separator}\hbox{\strut}}}]}

\newtheoremstyle{nonumberplainflex}% for proof environments, see below
	{\item[\theorem@headerfont\hskip\labelsep ##1\theorem@separator]}%
	{\item[\theorem@headerfont\hskip \labelsep ##3\theorem@separator]}

\newtheoremstyle{nonumberbreakflex}%
	{\item[\rlap{\vbox{\hbox{\hskip\labelsep \theorem@headerfont
		##1\theorem@separator}\hbox{\strut}}}]}%
	{\item[\rlap{\vbox{\hbox{\hskip\labelsep \theorem@headerfont 
		##3\theorem@separator}\hbox{\strut}}}]}

%% file: _preamble/ntheorem/da/_ntheorem_da_1.tex
\newcommand\danishtheorems{
	\def\@nthm@theorem{Sætning}
	\def\@nthm@proposition{Proposition}
	\def\@nthm@corollary{Korollar}
	\def\@nthm@lemma{Lemma}
	\def\@nthm@claim{Påstand}
	\def\@nthm@definition{Definition}
	\def\@nthm@example{Eksempel}
	\def\@nthm@exercise{Opgave}
	\def\@nthm@remark{Bemærkning}
	\def\@nthm@proof{Bevis}
}

%% file: _preamble/ntheorem/en/_ntheorem_en_1.tex
\newcommand\englishtheorems{
	\def\@nthm@theorem{Theorem}
	\def\@nthm@proposition{Proposition}
	\def\@nthm@corollary{Corollary}
	\def\@nthm@lemma{Lemma}
	\def\@nthm@claim{Claim}
	\def\@nthm@definition{Definition}
	\def\@nthm@example{Example}
	\def\@nthm@exercise{Exercise}
	\def\@nthm@remark{Remark}
	\def\@nthm@proof{Proof}
}

%% file: _preamble/hyperref/_hyperref_1.tex
\usepackage[hidelinks,breaklinks]{hyperref}

%% file: _preamble/cleveref/_cleveref_1.tex
\usepackage[nameinlink]{cleveref}

\ifdanish
	\input{_preamble/cleveref/da/_cleveref_da_1}
\else
	\input{_preamble/cleveref/en/_cleveref_en_1}
\fi

\crefalias{parabreak}{para}

\crefalias{theorembreak}{theorem}

\crefalias{definitionbreak}{definition}

\crefalias{examplebreak}{example}

\crefalias{corollarybreak}{corollary}

\crefalias{propositionbreak}{proposition}

\crefalias{lemmabreak}{lemma}

\crefalias{claimbreak}{claim}

\crefalias{remarkbreak}{remark}

\crefalias{textexercisebreak}{textexercise}

\newcommand{\crefrangeconjunction}{--} %gives e.g. "table~1--2" instead of "table 1 to~2"

%% file: _preamble/cleveref/da/_cleveref_da_1.tex
\crefname{para}{}{}
\crefname{theorem}{Sætning}{Sætningerne}
\crefname{definition}{Definition}{Definitionerne}
\crefname{example}{Eksempel}{Eksempler}
\crefname{corollary}{Korollar}{Korollarerne}
\crefname{proposition}{Proposition}{Propositionerne}
\crefname{lemma}{Lemma}{Lemmaerne}
\crefname{remark}{Bemærkning}{Bemærkningerne}
\crefname{claim}{Påstand}{Påstandene}
\crefname{exercise}{Opgave}{Opgaverne}
\crefname{textexercise}{Opgave}{Opgaverne}

\crefname{section}{afsnit}{afsnittene}
\crefname{equation}{ligning}{ligningerne}
\crefname{figure}{figur}{figurerne}
\crefname{table}{tabel}{tabellerne}
\crefname{page}{side}{siderne}
\crefname{part}{del}{delene}
\crefname{chapter}{kapitel}{kapitlerne}
\crefname{section}{afsnit}{afsnittene}
\crefname{appendix}{appendiks}{appendicer}
\crefname{enumi}{punkt}{punkterne}
\crefname{footnote}{fodnote}{fodnoterne}
\crefname{line}{linje}{linjerne}

%% file: _preamble/cleveref/en/_cleveref_en_1.tex
\crefname{para}{}{}
\crefname{theorem}{Theorem}{Theorems}
\crefname{definition}{Definition}{Definitions}
\crefname{example}{Example}{Examples}
\crefname{corollary}{Corollary}{Corollaries}
\crefname{proposition}{Proposition}{Propositions}
\crefname{lemma}{Lemma}{Lemmata}
\crefname{remark}{Remark}{Remarks}
\crefname{claim}{Claim}{Claims}
\crefname{exercise}{Exercise}{Exercises}
\crefname{textexercise}{Exercise}{Ecercises}

%% file: _preamble/theoremlists/_theoremlists_1.tex
\newlist{paralist}{enumerate}{2}
\setlist[paralist]{
	label=\textup{(\roman*)},
	ref=\thepara\textup{(\roman*)},
	resume,
}
\crefalias{paralisti}{theorem}

\newlist{theoremlist}{enumerate}{2}
\setlist[theoremlist]{
	label=\textup{(\roman*)},
	ref=\thepara\textup{(\roman*)},
	resume,
}
\crefalias{theoremlisti}{theorem}

\newlist{corollarylist}{enumerate}{2}
\setlist[corollarylist]{
	label=\textup{(\roman*)},
	ref=\thepara\textup{(\roman*)},
	resume,
}
\crefalias{corollarylisti}{corollary}

\newlist{propositionlist}{enumerate}{2}
\setlist[propositionlist]{
	label=\textup{(\roman*)},
	ref=\thepara\textup{(\roman*)},
	resume,
}
\crefalias{propositionlisti}{proposition}

\newlist{lemmalist}{enumerate}{2}
\setlist[lemmalist]{
	label=\textup{(\roman*)},
	ref=\thepara\textup{(\roman*)},
	resume,
}
\crefalias{lemmalisti}{lemma}

\newlist{definitionlist}{enumerate}{2}
\setlist[definitionlist]{
	label=\textup{(\roman*)},
	ref=\thepara\textup{(\roman*)},
	resume,
}
\crefalias{definitionlisti}{definition}

\newlist{textexerciselist}{enumerate}{2}
\setlist[textexerciselist]{
	label=\textup{(\roman*)},
	ref=\thepara\textup{(\roman*)},
	resume,
}
\crefalias{textexerciselisti}{textexercise}

\newlist{remarklist}{enumerate}{2}
\setlist[remarklist]{
	label=\textup{(\roman*)},
	ref=\thepara\textup{(\roman*)},
	resume,
}
\crefalias{remarklisti}{remark}

\newlist{examplelist}{enumerate}{2}
\setlist[examplelist]{
	label=\textup{(\roman*)},
	ref=\thepara\textup{(\roman*)},
	resume,
}
\crefalias{examplelisti}{example}

\newlist{exerciselist}{enumerate}{2}
\setlist[exerciselist]{
	label=\textup{(\roman*)},
	ref=\thepara\textup{(\roman*)},
	resume,
}
\crefalias{exerciselisti}{exercise}

\newlist{claimlist}{enumerate}{2}
\setlist[claimlist]{
	label=\textup{(\roman*)},
	ref=\thepara\textup{(\roman*)},
	resume,
}
\crefalias{claimlisti}{claim}

\newcounter{localreftmpcnt} %
\newcommand\alphsubformat[1]{\textup{(\roman{#1})}} %adapt ....
\newcommand\localref[2][\alphsubformat]{%
	\ifcsname r@#2@cref\endcsname
	\cref@getcounter {#2}{\mylabel}%
	\setcounter{localreftmpcnt}{\mylabel}%
	\hyperref[#2]{%
		\alphsubformat{localreftmpcnt}%
	}%
	\else ?? \fi}   

%% file: _doc_commands/_doc_commands_1.tex
\numberwithin{equation}{section}

\setupclass{category}{output=algmap}

\newmap\mconst{\operatorname{const}}

\setupclass{var}{
	singlekeys={
		{smash}{return,command=\noexpand\smash},
		{limits}{command=\mathop,return,symbolputright=\limits},
		{bin}{command=\mathbin},
		{operator}{command=\mathop},
	},
	keyvals={
		{powering}{upper={\keyvalue}},
		{latching}{symbolputleft=L_{\keyvalue}},
		{latch}{latching={\keyvalue}},
		{matching}{symbolputleft=M_{\keyvalue}},
		{match}{matching={\keyvalue}},
		{over}{symbolputright=/{\keyvalue}},
	},
	output=var,
}

\setupclass{homset}{
	singlekeys={{enrich}{command=\noexpand\underline}},
}

\newvar\valpha{\alpha}
\newvar\vbeta{\beta}
\newvar\vgamma{\gamma}
\newvar\va{a}
\newvar\vb{b}
\newvar\vc{c}
\newvar\vd{d}
\newvar\ve{e}
\newvar\vf{f}
\newvar\vg{g}
\newvar\vh{h}
\newvar\vv{v}
\newvar\vA{A}
\newvar\vB{B}
\newvar\vC{C}
\newvar\vD{D}
\newvar\vE{E}
\newvar\vX{X}
\newvar\vY{Y}
\newvar\vZ{Z}
\newvar\vR{R}
\newvar\vS{S}

\newmap\miota{\iota}
\newmap\mvarphi{\varphi}
\newmap\malpha{\alpha}
\newmap\mbeta{\beta}
\newmap\mtheta{\theta}
\newmap\mpsi{\psi}
\newmap\meta{\eta}
\newmap\mr{r}
\newmap\mW{W}
\newmap\comult{\Delta}
\newmap\counit{\eta}
\newmap\ca{\mathup{ca}}
\newmap\mult{m}
\newmap\mR{R}
\newmap\mS{S}

\newclassset\setA{A}
\newclassset\setB{B}
\newclassset\setC{C}
\newclassset\setX{X}
\newclassset\setY{Y}
\newclassset\setZ{Z}
\newclassset\setU{U}
\newclassset\setV{V}
\newclassset\setW{W}

\newalg\algG{G}
\newalg\algH{H}
\newalg\algA{A}
\newalg\algB{B}
\newalg\algC{C}
\newalg\algM{M}
\newalg\algN{N}
\newalg\algP{P}

\setupclass{var}{
	singlekeys={
		{*res}{upper=*,nopar,output=map},
		{!kan}{lower=!,nopar,output=map},
		{*kan}{lower=*,nopar,output=map},
		{qpull}{pull,nopar,output=alg},
	},
}

\newmapclass{algmap}[
	parent=alg,output=algmap,
]

\newalgmap\catA{\mathscr{A}}
\newalgmap\catB{\mathscr{B}}
\newalgmap\catC{\mathscr{C}}
\newalgmap\catD{\mathscr{D}}
\newalgmap\catE{\mathscr{E}}
\newalgmap\catV{\mathscr{V}}
\newalgmap\catGamma{\Gamma}

\setupclass{classset}{
	keyvals={
		{equivar}{upper={\keyvalue}},
	},
}

\newalgmap\catdelta{\mathbf{\Deltaup}}[
	singlekeys={
		{plus}{lower=+},
	},
]

\newcategory\catsset{SSet}
\newcategory\catvarsset{\catset}[diag=\catdelta]

\newcomplexclass{cosimplicial}[
	parent=var,
	gradingpos=upper,
	keyvals={
			{degreedefault}{sd={\keyvalue},ifoutput=true},
			{sk}{symbolputleft={\noexpand\operatorname{Sk}^{\keyvalue}}},
	},
	singlekeys={
			{deltaplus}{symbolputleft={\miota[*res]}},
	},
	output=classset,
]

\newcomplexclass{simplicial}[
	parent=cosimplicial,
	gradingpos=lower,
	singlekeys={
		{realize}{
			return,leftspar=\lvert,rightspar=\rvert,spar,nowoutput=classset,
		},
		{boundary}{
			symbolputleft=\partial,
		},
	},
	output=classset,
]

\newsimplicial\ssetX{X}
\newsimplicial\ssetY{Y}
\newsimplicial\ssetZ{Z}
\newsimplicial\ssetK{K}
\newsimplicial\ssetL{L}

\newcosimplicial\cosimpX{X}
\newcosimplicial\cosimpY{Y}

\newcosimplicial\simp{\Deltaup}[
	d={\!},iffirstd=true,
	output=simplicial,
]

\newcosimplicial\horn{\Lambdaup}[
	d={\!},iffirstd=true,output=simplicial,
]

\newalgmap\nerve{N}[
	output=simplicial,
]

\newdelim\ordset{[}{]}[
	output=alg,
	singlekeys={
		{plus}{outputoptions={upper=+}},
		{+}{plus},
	},
]

\setupclass{algop}{
	keyvals={
		{arg}{argwithoutkeyval={\keyvalue}},
	},
}

\newalgop\Prod{\prod}
\newalgop\Coprod{\coprod}
\newalgop\dirsum{\bigoplus}
\newautomathoperator\pushout{\smallcoprod}
\newautomathoperator\varpushout{\mathbin{\textstyle\coprod}}
\newautomathoperator\pullback{\smallprod}
\newautomathoperator\varpullback{\mathbin{\textstyle\prod}}
\newautomathoperator\bigpullback{\prod}
\newautomathoperator\varbigpullback{\bigtimes}
\newalgop\catob{\operatorname{Ob}}

\newalgop\Eq{\operatorname{Eq}}[par]
\newalgop\Coeq{\operatorname{Coeq}}[par]

\setupclass{alg}{
	keyvals={
		{modelstructure}{lower={\keyvalue}},
		{commacat}{return,symbolputright={/\keyvalue}},
	},
	singlekeys={
		{proj}{modelstructure={\modelstructureproj}},
		{inj}{modelstructure={\modelstructureinj}},
		{reedy}{modelstructure={\modelstructurereedy}},
		{reedydirect}{lower=+},
		{reedyinverse}{lower=-},
		{reedy+}{reedydirect},
		{reedy-}{reedyinverse},
	},
}

\newvariableclass{modelstructure}[parent=classset]

\newmodelstructure\modelstructureproj{\operatorname{Proj}}
\newmodelstructure\modelstructureinj{\operatorname{Inj}}
\newmodelstructure\modelstructurereedy{\operatorname{Reedy}}

\newcommand\noloc{%
  \nobreak
  \mspace{6mu plus 1mu}
  {:}
  \nonscript\mkern-\thinmuskip
  \mathpunct{}
  \mspace{2mu}
}

\newmap\mprojcof{\mathscr{F}}[output=map]

\newmapclass{replacementfunctor}[
	parent=algmap,
	singlekeys={
		{deltaplus}{
			lower=\catdelta[plus],
			output=simplicial,
		},
	},
]

\newreplacementfunctor\mcofrep{Q}
\newreplacementfunctor\mfibrep{R}
\newreplacementfunctor\cofrep{Q}[parent=mcofrep]
\newreplacementfunctor\fibrep{R}[parent=mfibrep]

\newalgop\End{\int}[
	output=alg,
	singlekeys={
		{rder}{symbolputleft={
			\noexpand\mathbb{R}
			\!\!\mathchoice{\!}{}{}{}
		}},
	},
]
\newalgop\Coend{\int}[
	gradingpos=lower,
	output=alg,
	arg={\mathchoice{\!\!\!}{\!}{\!}{\!}},
	iffirstarg=true,
	singlekeys={
		{lder}{symbolputleft={
			\noexpand\mathbb{L}
			\mathchoice{\!\!}{\!}{\!}{\!}
		}},
	},
]

\usetikzlibrary{patterns}
\usetikzlibrary{decorations.pathreplacing,decorations.markings}

\tikzset{
  on each segment/.style={
    decorate,
    decoration={
      show path construction,
      moveto code={},
      lineto code={
        \path [#1]
        (\tikzinputsegmentfirst) -- (\tikzinputsegmentlast);
      },
      curveto code={
        \path [#1] (\tikzinputsegmentfirst)
        .. controls
        (\tikzinputsegmentsupporta) and (\tikzinputsegmentsupportb)
        ..
        (\tikzinputsegmentlast);
      },
      closepath code={
        \path [#1]
        (\tikzinputsegmentfirst) -- (\tikzinputsegmentlast);
      },
    },
  },
  mid arrow/.style={postaction={decorate,decoration={
        markings,
        mark=at position .5 with {\arrow[#1]{stealth}}
      }}},
}

%% file: functors/functors_5016.tex
\begingroup

\begin{abstract}
	\noindent We prove and explain several classical
	formulae for homotopy (co)limits
	in general (combinatorial) model categories
	which are not necessarily simplicially enriched.
	Importantly, we prove versions of the Bousfield--Kan formula and the fat totalization formula in this complete generality. We finish with a proof that homotopy-final functors
	preserve homotopy limits, again in complete generality.
	
	\medskip
	
	\\model categories, simplicial sets, derived functors
	
	\medskip
	
\end{abstract}

\let\section=\chapter

\numberwithin{equation}{chapter}

\noindent If \( \catC \)~is a model category and \( \catGamma \)~a category,
we denote by \( \catC[diag=\catGamma] = \catfun{ \catGamma , \catC } \)
the category of functors~\( \catGamma\to\catC \), which we shall also refer to as “diagrams” of shape~\( \catGamma \).
It is natural to call a map of diagrams~\( \malpha \colon \mF \to \mG \)
in~\( \catC[diag=\catGamma] \) a \emph{weak equivalence}
if \( \malpha[\vgamma,smash] \colon \mF{\vgamma} \to \mG{\vgamma} \)
is a weak equivalence in~\( \catC \) for all objects~\( \vgamma \in \catGamma \).
We shall refer to such weak equivalences as \textdef{componentwise} weak equivalences.
But then we immediately run into the problem that the limit functor~\( \smash{ \invlim \colon \catC[diag=\catGamma] \to \catC } \)
does not in general preserve weak equivalences. Since \( \invlim \)~is a right adjoint, this leads us into trying to \emph{derive} it. The right derived functor of~\( \invlim \) is called the \textdef{homotopy limit} and is
denoted~\( \invholim \colon \catC[diag=\catGamma]\to\catC \).
Dually, the left derived functor of~\( \dirlim \)
is called the \textdef{homotopy colimit} and is denoted~\( \dirholim \).

For many purposes, the abstract existence of homotopy limits is all you need. However, there are also many cases where a concrete, minimalistic realization of them is useful for working with abstract notions.
%A concrete, minimalistic realization of homotopy limits like the one presented in this paper can sometimes be useful for working with abstract notions.
For instance, this paper grew out of an attempt to concretize a concept from derived algebraic geometry. More specifically, we wanted to develop a homological algebra model for the dg-category of quasi-coherent sheaves on a dg-scheme which are equivariant with respect to the action of a group dg-scheme.
This question was addressed in~\textcite{explicit} where a partial result was obtained under serious restrictions (Proposition~13).
The general case was stated as a conjecture, see~Conjecture~1 in the same paper.
In the companion paper to this one, \textcite{comodules},
we cover the general case and prove that conjecture (see~Theorem~4.1.1), and the key result of homotopical nature is proved in the present note
(see~\cref{ex:fat_totalization}).

Quillen's model category machinery tells us how to derive the limit:
We must equip the diagram category~\( \catC[diag=\catGamma] \)
with a model structure with componentwise weak equivalences and in which the
limit functor~\( \invlim\colon\catC[diag=\catGamma]\to\catC \) is a right Quillen functor.
In this case, the derived functor is given by
\(
	\invholim{ \mF } = \invlim { \mfibrep ( \mF ) }
\)
for some fibrant replacement~\( \mfibrep { \mF } \)
in~\( \catC[diag=\catGamma,smash] \).
Indeed, \cref{res:limit_functor_quillen} below shows that such a model structure on~\( \catC[diag=\catGamma] \) exists e.g.\ if the model category~\( \catC \) is \emph{combinatorial}.
More precisely, we introduce
%\begin{itemize}
%	\item The \textdef{projective model structure}~\( \catC[diag=\catGamma,proj] \)
%	where weak equivalences and fibrations are calculated componentwise.
%	\item The \textdef{injective model structure}~\( \catC[diag=\catGamma,inj] \)
%	where weak equivalences and cofibrations are calculated componentwise.
%\end{itemize}
the \textit{injective model structure}~\( \catC[diag=\catGamma,inj] \)
	where weak equivalences and cofibrations are calculated componentwise.
Denoting by \( \mconst \colon \catC \to \catC[diag=\catGamma] \)
the constant functor embedding,
we clearly see that
\(
	\mconst
	\colon
	\catC
	\rightleftarrows
	\catC[diag=\catGamma,inj,smash]
	\noloc
	\invlim
\)
is a Quillen adjunction since \( \mconst \)~preserves (trivial)
cofibrations.
%Dually,~\(
%	\dirlim
%	\colon
%	\catC[diag=\catGamma,proj,smash]
%	\rightleftarrows
%	\catC
%	\noloc
%	\mconst
%\)
%is a Quillen adjunction.

The injective model structure being in general rather complicated, calculating such a replacement of a diagram in practice becomes very involved for all but the simplest shapes of the category~\( \catGamma \). Therefore, traditionally, other tools have been used.
%One of the most popular ones comes from the observation that
%one may extend the limit
%functor~\( \invlim[\catGamma] \colon \catC[diag=\catGamma] \to \catC \)
%to obtain the \emph{end}
%functor~\(
%	\End[\catGamma,smash]
%	\colon
%	\catC[diag=\catGamma[op]\times\catGamma] \to \catC
%\), which we present below,
%i.e.\ we
%have~\( \invlim[\catGamma]{\mF} = \End[\catGamma]{\mF}[smash] \)
%when \( \mF \in \catC[diag=\catGamma] \)
%is regarded as a bifunctor in~\( \catC[diag=\catGamma[op]\times\catGamma] \)
%which is constant with respect to the first variable.
%This perspective on limits has the advantage that ends, despite their greater complexity, can be much easier to derive.
One of the most popular techniques involves
\emph{adding a parameter} to the limit functor~\( \invlim[\catGamma] \)
before deriving it. The result is the \emph{end}
bifunctor~\(
	\End[\catGamma]
	\colon
	\catGamma[op]\times\catGamma
	\to
	\catC
\)
(introduced below) which is in general much easier to derive.

One of the classical accounts of this technique is
\textcite{hir}
who mainly works in the setting of \emph{simplicial model categories}, which are model categories enriched
over simplicial sets, and which furthermore are equipped with
a \emph{powering} functor
\[
	\catsset[op]\times\catC
	\longto
	\catC,
	\qquad
	\tup{ \ssetK , \vc }
	\longmapsto
	\vc[powering=\ssetK],
\]
and a \emph{copowering} functor
\[
	\catsset\times\catC
	\longto
	\catC,
	\qquad
	\tup{ \ssetK , \vc }
	\longmapsto
	\ssetK\tens\vc,
\]
satisfying
some compatibility relations with the model structure.
%He defines homotopy limits by the classical \textdef{Bousfield--Kan formula}
He then establishes the classical \textdef{Bousfield--Kan formula}
\begin{equation}\label{eq:bousfield_kan}
	\invholim { \mF }
	=
	\End [ \vgamma\in\catGamma ]{
		{\mF{\vgamma}[powering=\nerve{{\catGamma[commacat=\vgamma]}}]}
	}
	,
\end{equation}
%where the integral refers to the so-called \emph{end} construction,
%and
where we write~\( \nerve{{ \catGamma[mod=\vgamma] }} \) for the
nerve of the comma category of maps in~\( \catGamma \)
with codomain~\( \vgamma \).
Or rather, he uses this formula as his \emph{definition}
of homotopy limits (see~definition~18.1.8).
%\citeauthor{hir} then generalizes this formula to arbitrary model categories, by first showing that in this more general setting,
%one may still obtain a weaker notion of powering and copowering,
%unique up to weak equivalence in a certain sense.
%An explanation of this formula (for simplicial model categories) is due to
A proof that this formula agrees with the general
definition of homotopy limits is due to
\textcite[equation~(2)]{gambino_simplicial}
using the machinery of Quillen 2-functors.

\citeauthor{hir} then generalizes this formula to arbitrary model categories in chapter~19, definition~19.1.5.
He shows that even for non-simplicial model categories,
one can replace simplicial
powerings and copowerings by a weaker notion,
unique up to homotopy in a certain sense.
He then again takes the formula~\eqref{eq:bousfield_kan}
to be his \emph{definition} of a homotopy limit.
This paper is devoted to proving why this formula
agrees with the general definition of a homotopy limit
(similarly to what \citeauthor{gambino_simplicial}
did in the simplicial setting).
To the authors' knowledge, such a proof has not been carried out in the
literature before.

\pagebreak

\section{The end construction}

Let \( \catGamma \) and~\( \catC \) be categories, \( \catC \)~complete and cocomplete, and let~\( \mH \colon \catGamma[op] \times \catGamma \to \catC \) a bifunctor. The \textdef{end} of~\( \mH \) is an object
\[
	\End[\catGamma]{ \mH } = \End[\vgamma \in \catGamma]{ \mH { \vgamma , \vgamma } }
\]
in~\( \catC \), together with morphisms~\(
	\End[\catGamma]{ \mH }
	\to
	\mH { \vgamma , \vgamma }
\) for all~\( \vgamma \in \catGamma \), such that for any~\( \mf \colon \vgamma \to \vgamma[prime] \), the following diagram commutes:
\begin{equation*}%\label{eq:end_universal_property}
\begin{tikzcd}
	\End[\catGamma]{ \mH } \ar[r] \ar[d]
			& \mH { \vgamma , \vgamma } \ar[d, "\mH { \vgamma , \mf }"]
\\
	\mH { \vgamma[prime] , \vgamma[prime] } \ar[r, "\mH { \mf , \vgamma[prime] }"']
			& \mH { \vgamma , \vgamma[prime] }
			\invdot
\end{tikzcd}
\end{equation*}
Furthermore, \( \End[\catGamma]{ \mH} \)~is universal with this property, meaning that if \( A \)~is another object of~\( \catC \) with a collection of arrows~\( A \to \mH { \vgamma , \vgamma } \) for all~\( \vgamma \), subject to the same commutativity conditions, then these factor through a unique arrow~\( A\to \End[\catGamma]{ \mH } \):
\begin{equation*}%\label{eq:end_universal_property}
\begin{tikzcd}
	A \ar[rd,dashed] \ar[rrd, bend left] \ar[ddr, bend right]
\\[-1em]
	&[-2em]	\End[\catGamma]{ \mH } \ar[r] \ar[d]
			& \mH { \vgamma , \vgamma } \ar[d, "\mH { \vgamma , \mf }"]
\\
	&	\mH { \vgamma[prime] , \vgamma[prime] } \ar[r, "\mH { \mf , \vgamma[prime] }"']
			& \mH { \vgamma , \vgamma[prime] }
			\invdot
\end{tikzcd}
\end{equation*}
%If \( \catC \)~is complete, then clearly
Clearly, we may obtain the end by the formula
\begin{equation*}\label{eq:Equalizer_formula_for_end}
	\End[\catGamma]{ H }
	=
	\Eq[par=\Big]{
		\Prod[ \vgamma \in \catGamma ]{ \mH { \vgamma , \vgamma } }
		\rightrightarrows
		\Prod[ \mathclap{\mf \colon \vgamma \to \vgamma[prime] } ]{ \mH { \vgamma , \vgamma[prime] } }
	}.
\end{equation*}
Here, the second product runs over all morphisms~\( \mf \colon \vgamma \to \vgamma[prime] \) in~\( \catGamma \), and the two arrows are given by
\(
	\mf[push]
	\colon 
	\mH { \vgamma , \vgamma }
	\to
	\mH { \vgamma , \vgamma[prime] }
\)
resp.~\(
	\mf[pull]
	\colon
	\mH { \vgamma[prime] , \vgamma[prime] }
	\to
	\mH { \vgamma , \vgamma[prime] }
\).
There is a dual notion of a \emph{coend}, denoted instead by~\(\smash{ \Coend[\catGamma]{ \mH }} \), which we shall not spell out.

\begin{example}
	Given functors~\( \mF , \mG \colon \catGamma \to \catGamma[prime] \), we obtain a bifunctor
	\[
		\mH = \Hom[ \catGamma[prime] ]{ \mF {\slot } , \mG {\slot } }
		\colon
		\catGamma[op] \times \catGamma \to \catset,
	\]
	and the universal property shows that
	\[
		\End[\catGamma]{ \mH }
		=
		\Hom[ \catfun { \catGamma , \catGamma[prime] } ]{ \mF , \mG }
	\]
	is the set of natural transformations between \( \mF \) and~\( \mG \).
\end{example}

\begin{example}
	A diagram~\( \mF \in \catC[diag=\catGamma] \) may be regarded as a diagram in~\( \catC[diag=\catGamma[op]\times\catGamma] \) which is
	constant with respect to the first variable. In that case, it follows from the universal property of the end that
	\( \End[\catGamma]{ \mF } = \invlim[\catGamma]{ \mF } \)
	recovers the limit of the diagram.
\end{example}

\begin{proposition}\label{res:end_adjunction}
	%Suppose \( \catC \)~is cocomplete.
	The end fits as the right adjoint of the adjunction
	\[\begin{tikzcd}[sep=scriptsize]
		\Coprod[ \Hom[\catGamma] ]
		\colon
		\catC \ar[r, yshift=1.5pt]
			& \catC[diag=\catGamma[op]\times\catGamma]
			\ar[l,yshift=-1.5pt]
			\noloc
			\End[\catGamma]\invdot
	\end{tikzcd}\]
	The left adjoint takes~\( \vA \in \catC \) to the bifunctor
	\( \Coprod[ \Hom[\catGamma]{\slot ,\slot } ]{ \vA } \colon \catGamma[op]\times\catGamma\to\catC \).
\end{proposition}

\begin{proof}
	Clear from the definition.
\end{proof}

This is equivalent to the statement that we have an adjunction
\[
	\catset[diag=\catGamma[op]\times\catGamma,par=\big]{
		\Hom[\catGamma] ,
		\catC { \vA , \mF }
	}
	\cong
	\catC[diag=\catGamma[op]\times\catGamma,par=\big]{
		\vA , \textstyle \End[\catGamma]{ \mF }
	}
\]
for \( \vA \in \catC \). This says that the end is the
\emph{weighted limit}~\(
	\catC[diag=\catGamma[op]\times\catGamma]
	\to
	\catC
\)
%\( \End[\catGamma,smash] = \invlim[weight=\Hom[\catGamma],\catGamma[op]\times\catGamma,smash] \)
with weight~\( \Hom[\catGamma] \).
The dual statement for \emph{coends} is that the
coend functor
\[\textstyle
	\Coend[\catGamma]
	\colon
	\catC[diag=\catGamma[op]\times\catGamma]
	\longto
	\catC
\]
is left adjoint to~\( \Prod[\Hom[\catGamma]] \).

\section{The projective and injective model structures}

If \( \catC \)~is a model category and \( \catGamma \)~any category, there is no completely general way to turn the functor category
\(
	\catC[diag=\catGamma]
	= \catfun{ \catGamma , \catC }
\)
into a model category. The na\"ive approach, calculating weak equivalences, cofibations, and fibrations componentwise, will not in general yield a model structure. It is natural to demand that at least the weak equivalences must be calculated componentwise for any model structure to be satisfactory. In general, however, at least one of the other two classes will in return become more complicated. The two most natural model structures one can hope for (which may or may not exist) are
\begin{itemize}
	\item The \textdef{projective model structure}~%
	%\( \tup { {\catC[diag={\catGamma}]} , \modelstructureproj } \)
	\( \catC[diag=\catGamma,proj] \)
	where weak equivalences and fibrations are calculated componentwise.
	\item The \textdef{injective model structure}~%
	%\( \tup { {\catC[diag={\catGamma}]} , \modelstructureinj } \)
	\( \catC[diag=\catGamma,inj] \)
	where weak equivalences and cofibrations are calculated componentwise.
\end{itemize}
Existence of these model structures depends heavily on the structure of the target category~\( \catC \) (see \cref{res:existence_of_proj_inj_model_structures} below).
We shall also use the attributes \textquote{projective(ly)} and \textquote{injective(ly)} when referring to these model structures, so e.g.~\textquote{projectively cofibrant} means cofibrant in the projective model structure.

\begin{propositionbreak}[Proposition {\parencite[Proposition~A.2.8.7]{htt}}]%
	\label{res:Kan_extensions_Quillen_adjunctions}
	If \( \catC \)~is a model category and \( \mf \colon \catGamma\to\catGamma[prime] \) a functor,
	%then the pullback
	denote by~\( \mf[*res] \) the restriction functor
	\(\smash{
		\catC[diag={\catGamma[prime]}]
		\to
		\catC[diag={\catGamma}]
	}\).
	Then \( \mf[*res] \)~fits as the right and left adjoint of Quillen adjunctions
	\[
	\begin{tikzcd}[sep=scriptsize]
		\mf[!kan]
		\colon
		\catC[diag=\catGamma,proj]
		\ar[r,yshift=1.5pt]
		&
		\ar[l,yshift=-1.5pt]
		\catC[diag=\catGamma[prime],proj]
		\noloc
		\mf[*res]
	\end{tikzcd}
		\qquad\text{resp.}\qquad
	\begin{tikzcd}[sep=scriptsize]
		\mf[*res]
		\colon
		\catC[diag=\catGamma[prime],inj]
		\ar[r,yshift=1.5pt]
		&
		\ar[l,yshift=-1.5pt]
		\catC[diag=\catGamma,inj]
		\noloc
		\mf[*kan]
	\end{tikzcd}
	\]
	whenever the model structures in question exist.
\end{propositionbreak}

The adjoints \( \mf[!kan] \) and~\( \mf[*kan] \) are the usual \textdef[left Kan extension]{left} and \textdef[right Kan extension]{right Kan extensions} along~\( \mf \), which are given by
%weighted
limits
\begin{equation}\label{eq:kan_extensions}%\textstyle
	\mf[!kan]{ \mF }{ \vgamma[prime] }
	=
	\dirlim[\mf{\vgamma}\to\vgamma[prime]]{\mF{\vgamma}}
	\qquad\text{and}\qquad
	\mf[*kan]{ \mF }{ \vgamma[prime] }
	=
	\invlim[\vgamma[prime]\to\mf{\vgamma}]{\mF{\vgamma}}
	.
\end{equation}
These limits are taken over the categories of maps~\( \mf{\vgamma}\to\vgamma[prime] \) (resp.~\( \vgamma[prime]\to\mf{\vgamma} \))
in~\( \catGamma[prime] \)
for varying~\( \vgamma \in \catGamma \).

\begin{proof}
	Since the adjunctions in question exist, their being Quillen follows from the observation that \( \mf[*res] \)~clearly preserves (trivial) projective fibrations and (trivial) injective cofibrations.
\end{proof}

\newvar\vc{c}

\begin{corollary}
	Assume in the following that the relevant model structures exist.
	\begin{corollarylist}
		\item\label{res:simple_cofibrations} If \( \mvarphi \colon \vc \to \vc[prime] \) is a (trivial) cofibration in~\( \catC \) and~\( \vgamma[0] \in \catGamma \) is an object, then the coproduct map
		\(
			\Coprod[\catGamma{ \vgamma[0] ,\slot }]{ \mvarphi }
			\colon
			\Coprod[\catGamma{ \vgamma[0] ,\slot }]{ \vc }
			\to
			\Coprod[\catGamma{ \vgamma[0] ,\slot }]{ \vc[prime] }
		\)
		is a (trivial) cofibration in~%
		%\( \tup { { \catC[diag=\catGamma,smash] } , \modelstructureproj } \).
		\( \catC[diag=\catGamma,proj,smash] \).
		We shall refer to such (trivial) cofibrations as \textdef[simple (trivial) projective cofibration]{simple projective cofibrations}\index{cofibration!projective!simple}.
		\item\label{res:preserve_simple_cofibrations} If \( \mf \colon \catGamma \to \catGamma[prime] \)~is a functor, then
		\(
			\mf[!kan]
			\colon
			%\tup { {\catC[diag={\catGamma}]} , \modelstructureproj }
			\catC[diag={\catGamma},proj]
			\to
			%\tup { {\catC[diag={\catGamma[prime]}]}, \modelstructureproj }
			\catC[diag={\catGamma[prime]},proj]
		\)
		preserves simple (trivial) projective cofibrations,
		taking \( \Coprod[\catGamma{ \vgamma[0] ,\slot }]{ \mvarphi } \)
		to~%
		\( \Coprod[ { \catGamma[prime]{ \mf { \vgamma[0] } ,\slot } } ]{ \mvarphi } \).
		\item If \( \mpsi \colon \vc \to \vc[prime] \)
		is a (trivial) fibration in~\( \catC \) and \( \vgamma[0] \in \catGamma \)~is an object,
		then the product map
		\(
			\Prod[ \catGamma {\slot , \vgamma[0] } ]{ \mpsi }
			\colon
			\Prod[ \catGamma {\slot , \vgamma[0] } ]{ \vc }
			\to
			\Prod[ \catGamma {\slot , \vgamma[0] } ]{ \vc[prime] }
		\)
		is a (trivial) fibration in~%
		\(
			%\tup { { \catC[diag=\catGamma] } , \modelstructureinj }
			\catC[diag=\catGamma,inj,smash]
		\).
		We shall refer to such (trivial) fibrations as \textdef[simple (trivial) injective fibration]{simple injective fibrations}\index{fibration!injective!simple}.
		\item If \( \mf \colon \catGamma \to \catGamma[prime] \) is a functor,
		then
		\(
			\mf[*kan]
			\colon
			%\tup { {\catC[diag={\catGamma}]} , \modelstructureinj }
			\catC[diag={\catGamma},inj]
			\to
			%\tup { {\catC[diag={\catGamma[prime]}]}, \modelstructureinj }
			\catC[diag={\catGamma[prime]},inj]
		\)
		preserves simple (trivial) injective fibrations,
		taking \( \Prod[ \catGamma {\slot , \vgamma[0] } ]{ \mpsi } \)
		to~\( \Prod[ { \Gamma'( {\slot , \mf { \vgamma[0] } } ) } ]{ \mpsi } \).
	\end{corollarylist}
\end{corollary}

\begin{proof}
	%The statement~\localref{res:simple_cofibrations}
	%follows by applying~\cref{res:Kan_extensions_Quillen_adjunctions}
	%to the embedding~\( \vgamma[0] \into \catGamma \) of the full subcategory with~\( \vgamma[0] \) as the only object and using the colimit formula for Kan extensions.
	Applying~\cref{res:Kan_extensions_Quillen_adjunctions} to the embedding%
	~\( \miota \colon \vgamma[0] \into \catGamma \) of the full subcategory with~\( \vgamma[0] \) as the only object,
	we get that \( \miota[!kan]{ \mvarphi } \)~is a (trivial) cofibration.
	Now
	\(
		\miota[!kan]{ \mvarphi }
		=
		\Coprod[ { \catGamma { \vgamma[0],\slot } } , smash ]{ \mvarphi }
	\)
	by the above colimit formula for left Kan extension.
	The statement~\localref{res:preserve_simple_cofibrations}
	follows %because Kan extensions (being adjoints to restriction) respect compositions
	by applying Kan extensions to the diagram
	\[\begin{tikzcd}[sep=small]
		\vgamma[0] \ar[r,hook] \ar[d]	& \catGamma \ar[d,"\mf"]
	\\
		\mf { \vgamma[0] } \ar[r,hook]	& \catGamma[prime]
	\end{tikzcd}\]
	and using that Kan extensions, being adjoints to restriction, respect compositions. The other statements are dual.
\end{proof}

\begin{corollary}\label{res:limit_functor_quillen}
	Denote by \( \mconst \colon \catC \to \catC[diag=\catGamma] \) the functor taking~\( \vc \in \catC \) to the constant diagram at~\( \vc \).
	\begin{corollarylist}
		\item\label{res:limit_functor_quillen_dirlim} If
		%\( \tup { { \catC[diag=\catGamma] } , \modelstructureproj } \)
		\( \catC[diag=\catGamma,proj] \)
		exists, then
		\(
			\dirlim
			\colon
			%\tup { { \catC[diag=\catGamma] } , \modelstructureproj }
			\catC[diag=\catGamma,proj]
			\rightleftarrows
			\catC
			\noloc
			\mconst
		\)
		is a Quillen adjunction.
		\item\label{res:limit_functor_quillen_invlim} If
		%\( \tup { { \catC[diag=\catGamma] } , \modelstructureinj } \)
		\( \catC[diag=\catGamma,inj] \)
		exists, then
		\(
			\mconst
			\colon
			\catC
			\rightleftarrows
			%\tup { { \catC[diag=\catGamma] } , \modelstructureinj }
			\catC[diag=\catGamma,inj]
			\noloc
			\invlim
		\)
		is a Quillen adjunction.
	\end{corollarylist}
\end{corollary}

\begin{proof}
	Apply \cref{res:Kan_extensions_Quillen_adjunctions} to the functor \( \catGamma\to * \).
\end{proof}

\begin{propositionbreak}[Proposition {\parencite[Proposition~A.2.8.2]{htt}}]%
	\label{res:existence_of_proj_inj_model_structures}
	If \( \catC \)~is a combinatorial model category, both the projective and injective model structures on~\( \catC[diag=\catGamma,smash] \) exist and are combinatorial.
	
	Given a generating set of (trivial) cofibrations in~\( \catC \), the corresponding \emph{simple} (trivial) cofibrations in~\( \catC[diag=\catGamma,proj,smash] \), for all choices
	of~\( \vgamma[0] \in \catGamma \), form a generating set of cofibrations.
\end{propositionbreak}

\begin{proof}[Sketch of proof]
	For the projective model structure, one checks by hand that simple (trivial) cofibrations have the left lifting property with respect to all degreewise fibrations (trivial fibrations).
%	One then takes a generating set of cofibrations in~\( \catC \) and shows that the corresponding simple cofibrations, for all choices of objects~\( \vgamma[0] \in \catGamma \), form a generating set of cofibrations for~%
%	\(
%		%\tup { { \catC[diag=\catGamma] } , \modelstructureproj }
%		\catC[diag=\catGamma,proj,smash]
%	\).
	One then checks that the mentioned simple (trivial) cofibrations form a generating set.
	The injective model structure, on the other hand, requires more work and has a less explicit set of generating cofibrations.
\end{proof}

\begin{propositionbreak}[Proposition {{\parencite[Remark~A.2.8.6]{htt}}}]%
	\label{res:projective_injective_quillen_functorial}
	A Quillen adunction \(
		\mF \colon \catC \rightleftarrows \catD \noloc \mG
	\)
	between combinatorial model categories induces Quillen adjunctions
	\[
	\begin{tikzcd}[sep=scriptsize]
		\catC[diag=\catGamma,proj]
		\ar[r,yshift=1.5pt]
		&
		\ar[l,yshift=-1.5pt]
		\catD[diag=\catGamma,proj]
	\end{tikzcd}
		\qquad\text{and}\qquad
	\begin{tikzcd}[sep=scriptsize]
		\catC[diag=\catGamma,inj]
		\ar[r,yshift=1.5pt]
		&
		\ar[l,yshift=-1.5pt]
		%\rightleftarrows
		\catD[diag=\catGamma,inj]
	\end{tikzcd}
	\]
	which are Quillen equivalences if \( \tup { \mF , \mG } \)~are.
\end{propositionbreak}

\section{The Reedy model structure}

A third approach exists to equip diagram categories~\( \catC[diag=\catGamma] \) with a model structure, provided the category~\( \catGamma \) has the structure of a Reedy category. Remarkably, unlike the projective and injective cases, this does not rely on any internal structure of~\( \catC \).

A category~\( \catGamma \) is called \textdef{Reedy} if it contains two subcategories
\( \catGamma[reedy+] , \catGamma[reedy-] \subset \catGamma \),
each containing all objects, such that
\begin{itemize}
	\item there exists a degree function
	\( \catob{\catGamma} \to \Z \),
	such that non-identity morphisms from~\( \catGamma[reedy+] \) strictly raise the degree
	and non-identity morphisms from~\( \catGamma[reedy-] \) strictly lower the degree (more generally, an ordinal number can be used instead of~\( \Z \));
	\item each morphism \( \mf \in \catGamma \) factors \emph{uniquely} as \( \mf = \mg \mh \) for \( \mg \in \catGamma[reedy+] \) and
	\( \mh \in \catGamma[reedy-] \).
\end{itemize}

We note that a direct category is Reedy with
\( \catGamma[reedy+] = \catGamma \),
and that an inverse category is Reedy with
\( \catGamma[reedy-] = \catGamma \).

\begin{remark}
	If \( \catGamma \)~is Reedy, then so is~\( \catGamma[op] \),
	with
	\( \catGamma[op,spar,reedy+] = \catGamma[reedy-,spar,op] \)
	and
	\( \catGamma[op,spar,reedy-] = \catGamma[reedy+,spar,op] \).
\end{remark}

\begin{example}\label{ex:delta_reedy}
	The simplex category~\( \catdelta \) is Reedy with~\( \catdelta[reedy+] \) consisting of injective maps and \( \catdelta[reedy-] \)~consisting of surjective maps.
	The degree function does the obvious thing,
	\( \ordset{n} \mapsto n \).
\end{example}

If \( \catGamma \)~is a Reedy category and \( \catC \)~is any model category, and if \( \mF \in \catC[diag=\catGamma] \)~is a diagram, we define the \textdef[latching object]{latching} and \textdef[matching object]{matching objects} by
\[
	\mF[latch=\vgamma]
	= \dirlim[ (\valpha\underset{\smash{\neq}}{ \to }\vgamma)\in\catGamma[reedy+] ]{ \mF { \valpha } }
\qquad\text{and}\qquad
	\mF[match=\vgamma]
	= \invlim[ (\vgamma\underset{\smash{\neq}}{\to}\valpha)\in\catGamma[reedy-] ]{ \mF { \valpha } }.
\]
In other words, the limit (resp.,\ colimit) runs over the category of all \emph{non-identity} maps \( \valpha \to \vgamma \) in~\( \catGamma[reedy+] \) (resp.,\ \( \vgamma \to \valpha \) in~\( \catGamma[reedy-] \)).
The \textdef{latching map} is the canonical map
\( \mF[latch=\vgamma,smash] \to \mF{\vgamma} \),
and the \textdef{matching map} is the canonical map
\( \mF{\vgamma} \to \mF[match=\vgamma,smash] \).

If \( \mf \colon \mF \to \mG \) is a map in~\( \catC[diag=\catGamma] \), then the \textdef{relative latching map} is the map
\[
	\mF { \vgamma }
	\varpushout[ limits, {\mF[latch=\vgamma]} ]
	\mG[latch=\vgamma]
	\longto
	\mG { \vgamma }
\]
given by the universal property of the pushout. We say that \( \mf \)~is a \textdef{(trivial) Reedy cofibration} if the relative latching map is a (trivial) cofibration in~\( \catC \).
If \( \mF = \varnothing \), we recover the latching map.
Dually, the \textdef{relative matching map}
is the map
\[
	\mF { \vgamma }
	\longto
	\mG { \vgamma }
	\varpullback[ limits, {\mG[match=\vgamma]} ]
	\mF[match=\vgamma]
\]
given by the universal property of the pullback. We say that \( \mf \)~is a \textdef{(trivial) Reedy fibration} if the relative matching map is a (trivial) fibration in~\( \catC \).
If \( \mF = * \), we recover the matching map.
%\fxfatal{Think about defining (absolute) latching/matching maps.}

\begin{propositionbreak}[{Proposition
	\parencite[Theorem~15.3.4]{hir}}]
	If \( \catC \)~is an arbitrary model category and \( \catGamma \)~is a Reedy category, then this defines a model structure on~\( \catC[diag=\catGamma,smash] \), called the \textdef{Reedy model structure}.
	The weak equivalences are componentwise weak equivalences.
	We shall write~\( \catC[diag=\catGamma,reedy,smash] \) when we equip the diagram category with this model structure.
\end{propositionbreak}

\begin{propositionbreak}[{Proposition
	\parencite[Example~A.2.9.22]{htt}}]%
	\label{res:reedy_projective_injective_restriction}
	Let~\( \catC \) be a model category and \( \catGamma \)~a Reedy category.
	Then
	\begin{propositionlist}
		\item If \( \catGamma = \catGamma[reedy+] \)~is a direct category, the projective model structure~\( \catC[diag=\catGamma,proj] \)
		exists and coincides with the Reedy model structure.
		\item If \( \catGamma = \catGamma[reedy-] \)~is an inverse category,
		the injective model structure~\( \catC[diag=\catGamma,inj] \)
		exists and coincides with the Reedy model structure.
	\end{propositionlist}
	Furthermore, a map \( \mf \colon \mF \to \mG \) in~\( \catC[diag=\catGamma] \) is a
	\begin{propositionlist}[resume]
		\item\label{res:reedy_projective_injective_restriction_plus} (trivial) cofibration if and only if the restriction
		\( \smash{
			\mf[res=\catGamma[reedy+]]
			\colon
			\mF[res=\catGamma[reedy+]]
			\to
			\mG[res=\catGamma[reedy+]]
		} \)
		is a (trivial) projective cofibration in~\(
			\catC[diag=\catGamma[reedy+],proj,smash]
		\).
%		In particular, if \( \catGamma = \catGamma[reedy+] \) is direct,
%		then the projective model structure~\(\catC[diag=\catGamma,proj,smash] \)
%		exists and coincides with the Reedy model structure.
		\item\label{res:reedy_projective_injective_restriction_minus} (trivial) fibration if and only if the restriction
		\( \smash{
			\mf[res=\catGamma[reedy-]]
			\colon
			\mF[res=\catGamma[reedy-]]
			\to
			\mG[res=\catGamma[reedy-]]
		} \)
		is a (trivial) injective fibration in~\(
			\catC[diag=\catGamma[reedy-],inj,smash]
		\).
%		In particular, if \( \catGamma = \catGamma[reedy-] \) is inverse,
%		then the injective model structure~\( \catC[diag=\catGamma,inj,smash] \)
%		exists and coincides with the Reedy model structure.
	\end{propositionlist}
\end{propositionbreak}

\begin{propositionbreak}[{Proposition
	\parencite[Theorem~15.5.2]{hir}}]%
	\label{res:reedy_category_product}
	If \( \catGamma \) and \( \catGamma[prime] \)~are both Reedy categories, then so is \( \catGamma\times\catGamma[prime] \), and
	the three possible Reedy model structures one can put on
	\( \catC[diag=\catGamma\times\catGamma[prime]] \)
	agree, i.e.
	\[
		\catC[diag=\catGamma\times\catGamma[prime],reedy]
		=
		\catC[diag=\catGamma,reedy,spar=\big,smash,diag=\catGamma[prime],reedy]
		=
		\catC[diag=\catGamma[prime],reedy,spar=\big,smash,diag=\catGamma,reedy]
		.
	\]
\end{propositionbreak}

%\newpage

\section{Homotopy limits}

The following theorem is the basis for all our homotopy limit
formulae:

\begin{theorem}\label{res:end_functor_quillen}
	Let~\( \catC \) be a model category and~\( \catGamma \) a category. Regard the functor category~\( \catC[diag=\catGamma[op]\times\catGamma] \) as a model category in any of the following ways:
	\begin{theoremlist}
		\item\label{res:end_functor_quillen_gammaop_gamma} as~\(
			\catC[diag=\catGamma[op]\times\catGamma]
			= \catC[
				diag=\catGamma[op],
				proj,
				spar,
				diag=\catGamma,
				inj,
			]
		\) (assuming this model structure exists);
		\item\label{res:end_functor_quillen_gamma_gammaop} as~\(
			\catC[diag=\catGamma[op]\times\catGamma]
			= \catC[
				diag=\catGamma,
				proj,
				spar,
				diag=\catGamma[op],
				inj,
			]
		\) (assuming this model structure exists);
		\item\label{res:end_functor_quillen_reedy} as~\(
			\catC[diag=\catGamma[op]\times\catGamma]
			= \catC[
				diag=\catGamma[op]\times\catGamma,
				reedy,
			]
		\) (assuming \( \catGamma \)~is Reedy).
	\end{theoremlist}
	Then the end functor~\(
		\End[\catGamma]
		\colon
		\catC[diag=\catGamma[op]\times\catGamma]
		\to
		\catC
	\)
	is right Quillen.
\end{theorem}

\begin{proof}
	We initially prove the first statement and obtain the second one by duality.
	By \cref{res:end_adjunction}, it suffices to check that the left
	adjoint~\( \Coprod[\Hom[\catGamma]] \)
	takes (trivial) cofibations in~\( \catC \) to (trivial) cofibations in~\(
		\catC[
			diag=\catGamma[op],
			modelstructure=\modelstructureproj,
			spar,
			diag=\catGamma,
			modelstructure=\modelstructureinj,
			smash,
		]
	\).
	If \( \vc \to \vc[prime] \)~is a (trivial) cofibration in~\( \catC \), then
	we must therefore consider the map~\(
		\Coprod [ \catGamma {\slot ,\slot } ]{ \vc }
		\to
		\Coprod [ \catGamma {\slot ,\slot } ]{ \vc[prime] }
	\)
	in~\(
		\catC[
			diag=\catGamma[op],
			modelstructure=\modelstructureproj,
			spar,
			diag=\catGamma,
			modelstructure=\modelstructureinj,
			smash,
		]
	\).
	Checking that this is a (trivial) injective cofibration
	over~\( \catGamma \)
	amounts, by definition, to checking this componentwise. But for a fixed~\( \vgamma[0] \in\catGamma \), this component is
	\(
		\Coprod[ \catGamma{\slot , \vgamma[0] } ]{ \vc }
		\to
		\Coprod[ \catGamma{\slot , \vgamma[0] } ]{ \vc[prime] },
	\)
	which is a simple (trivial) projective cofibration
	in~\( \catC[diag=\catGamma[op],smash] \).
	
	For the Reedy case, we recall from
	\cref{res:reedy_projective_injective_restriction,res:reedy_category_product}
	that being a (trivial) cofibration in the model category~\(
		\catC[diag=\catGamma\times\catGamma[op],reedy,smash]
		= \catC[diag=\catGamma,reedy,spar,diag=\catGamma[op],reedy,smash]
	\)
	is equivalent to the restriction being a (trivial) cofibration
	in
	\[
		\catC[
			diag={\catGamma[reedy+]},proj,
			spar=\big,smash,
			diag={\catGamma[op,spar,reedy+]},proj,
			%smash,
		]
		=
		\catC[
			diag={\catGamma[reedy+]},proj,
			spar=\big,smash,
			diag={\catGamma[reedy-,spar,op]},proj,
			%smash,
		].
	\]
	But we have, by the unique factorization property of Reedy categories, that
	\[
		\Coprod[ \catGamma{\slot ,\slot } ]{ \vc }
		= \Coprod[ \vgamma[0] \in \catGamma ]{
			\Coprod [ \catGamma[reedy-] ({\slot , \vgamma[0] }) ]{
				\Coprod [ \catGamma[reedy+] ({ \vgamma[0] ,\slot }) ]{ \vc }
			}
		}
	\]
	for any~\( \vc\in\catC \).
	These consist of coproducts of exactly the same form as the ones appearing in the definition of simple (trivial) projective cofibrations
	(\cref{res:simple_cofibrations}).
	Thus we find that for any (trivial) cofibration~\( \vc\to\vc[prime] \) in~\( \catC \),
	the map~\(
		\Coprod[ \catGamma {\slot ,\slot } ]{ \vc }
		\to
		\Coprod[ \catGamma {\slot ,\slot } ]{ \vc[prime] }
	\)
	is a (trivial) cofibration
	in~\( \catC[diag=\catGamma[op]\times\catGamma,reedy,smash] \).
\end{proof}

Thus we can derive the end using any of these three model structures, when available. Write \( \End[rder,\catGamma] \colon\catC[diag=\catGamma[op]\times\catGamma]\to\catC \) for the derived functor, which we shall call the \textdef{homotopy end}.

\begin{corollary}\label{res:homotopy_limits_via_ends}
	If \( \catC \)~is a combinatorial model category and \( \catGamma \)~a category, then for a diagram~\( \mF \in \catC[diag=\catGamma] \),
	\[\textstyle
		\invholim[\catGamma]{ \mF } = \End[rder,\catGamma]{ \mF }
		= \End [ \catGamma]{ \mfibrep{ \mF } }
		,
	\]
	where \( \mfibrep \)~is a fibrant replacement with respect to the model structure~\(
		\catC[diag=\catGamma,proj,spar,diag=\catGamma[op],inj,smash]
	\) or, if \( \catGamma \)~is Reedy,
	in~\(
		\catC[diag=\catGamma[op]\times\catGamma,reedy,smash]
	\).
\end{corollary}

\begin{proof}
	First write
	\(
		\invholim { \mF }
		=
		\invlim { \mfibrep[\catGamma]{ \mF } }
	\)
	for some fibrant replacement functor~\( \mfibrep[\catGamma] \)
	in~\( \catC[diag=\catGamma,inj,smash] \).
	Now \cref{res:limit_functor_quillen_dirlim,res:projective_injective_quillen_functorial} show that the constant functor embedding~\(
		\catC[diag=\catGamma,inj,smash]
		\into
		\catC[diag=\catGamma[op],proj,spar,diag=\catGamma,inj,smash]
	\)
	is right Quillen and thus preserves fibrant objects.
	Thus \( \mfibrep[\catGamma]{ \mF } \)~is also fibrant
	in~\( \catC[diag=\catGamma[op],proj,spar,diag=\catGamma,inj,smash] \).
	This proves the first equality sign. The second one is clear.
\end{proof}

Of course, even though \( \mF \) as a diagram in~\( \catC[diag=\catGamma[op]\times\catGamma] \) was constant with respect to the first variable, \( \mfibrep{ \mF } \) is in general not. Remarkably, since ends calculate naturality between the two variables, this often makes calculations of homotopy limits more manageable, compared to resolving the diagram inside~\( \catC[diag=\catGamma,inj,smash] \).

%\begin{examplebreak}[Example: Totalization formula]
%	
%\end{examplebreak}

\begin{corollary}\label{res:holim_direct_category}
	Suppose \( \catGamma \)~is a direct category,
	and let~\(
		\smash{
		\mfibrep
		\colon
		\catC
		\to
		%\catC[diag=\catGamma[op],reedy]
		%=
		\catC[diag=\catGamma[op],inj]
		}
	\)
	be a functor that takes~\( \vc\in\catC \) to a fibrant replacement of the constant diagram at~\( \vc \).
	Then
	\[\textstyle
		\invholim[\catGamma]{ \mF }
		=
		\End[\vgamma\in\catGamma]{
			\mfibrep{ \mF{\vgamma} }(\vgamma)
		}.
	\]
\end{corollary}

\begin{proof}
	Clearly, \( \fibrep{ \mF } \)~is a
	fibrant replacement
	inside~\(
		\catC[
			diag=\catGamma[op],
			inj,
			spar,
			diag=\catGamma,
			proj,
			smash,
		]
	\).
	By \cref{res:reedy_projective_injective_restriction,res:reedy_category_product},
	this model category is equal
	to
	\[
		\catC[
			diag=\catGamma[op],
			inj,
			spar=\big,
			smash,
			diag=\catGamma,
			proj,
			%smash,
		]
		=
		\catC[
			diag=\catGamma[op],
			reedy,
			spar=\big,
			smash,
			diag=\catGamma,
			reedy,
			%smash,
		]
		=
		\catC[
			diag=\catGamma,
			reedy,
			spar=\big,
			smash,
			diag=\catGamma[op],
			reedy,
			%smash,
		]
		=
		\catC[
			diag=\catGamma,
			proj,
			spar=\big,
			smash,
			diag=\catGamma[op],
			inj,
			%smash,
		]
	,
	\]
	so the result follows from~\cref{res:end_functor_quillen}.
\end{proof}

\section{Bousfield--Kan formula}

In \textcite[chapter~19]{hir}, homotopy limits are being developed
for arbitrary model categories via a machinery of simplicial resolutions. In this section, we use \cref{res:end_functor_quillen}/\cref{res:homotopy_limits_via_ends} to explain why this machinery works.
Throughout, we denote by~\( \catsset \)
the category of simplicial sets endowed with the Quillen model structure.

If \( \catC \)~is a (complete) category and~\( \ssetX{*} \in \catC[diag=\catdelta[op]] \)
a simplicial diagram in~\( \catC \), we may extend~\( \ssetX{*} \)
to a continuous functor~\( \ssetX \colon \catsset[op] \to \catC \)
via the right Kan extension
along
the Yoneda embedding~\(
	\catdelta[op] \into \catsset[op]
\):
\[
	\ssetX[powering=\ssetK]
	=
	\invlim[\simp{n}\to\ssetK]{ \ssetX{n} },
	\qquad
	\ssetK\in\catsset.
\]
If \( \catC \)~is a model category, the matching object at~\( \ordset{n} \)
is~\( \ssetX[match=n]{*} = \cosimpX[powering=\simp{n}[boundary]] \),
and so \( \ssetX{*} \)~being Reedy-fibrant
is equivalent
to the map~\(\smash{
	\ssetX{n}
	=
	\ssetX[powering=\simp{n}]
	\to
	\ssetX[powering=\simp{n}[boundary]]
}\)
being a fibration in~\( \catC \) for all~\( n \).

%\begin{lemma}
%	Suppose \( \catC \)~is a combinatorial model category and~\( \Gamma \)
%	a category, and let~\(
%		\mF
%		\in
%		\catC[
%			diag=\catdelta[op],
%			reedy,
%			spar,
%			diag=\catGamma,
%			proj,
%		]
%	\)
%	be fibrant.
%\end{lemma}

\begin{theorem}[Theorem (Bousfield--Kan formula)]\label{res:bousfield_kan}
	Suppose \( \catC \)~is a combinatorial model category,~\( \Gamma \)
	a category, and~\( \mF \in \catC[diag=\catGamma] \).
	Let~\(
		\fibrep
		\colon
		\catC
		\to
		\catC[diag=\catdelta[op],smash]
	\)
	be a functor that takes~\( \vc \in \catC \)
	to a Reedy-fibrant replacement of the constant \( \catdelta[op] \)-diagram at~\( \vc \).
	Let furthermore~\( \ssetK \in \catsset[diag=\catGamma,proj,smash] \)
	be a projectively cofibrant resolution of the point.
	Then
	\[
		{\textstyle\invholim[\catGamma]{ \mF }}
		=
		\End[\vgamma\in\catGamma]{
			\fibrep[output=simplicial]{\mF{\vgamma}}[
				powering={\ssetK[arg=\vgamma]},
			]
		}.
	\]
\end{theorem}

One may prove
\parencite[see e.g.][Proposition~14.8.9]{hir}
that the diagram~\( \ssetK[arg=\slot] = \nerve{{ \catGamma[mod=\slot] }} \in \catsset[diag=\catGamma,proj,smash] \),
taking~\( \vgamma \)
to the nerve~\( \nerve{{ \catGamma[mod=\vgamma] }} \)
of the comma category~\( \catGamma[mod=\vgamma] \) of all maps in~\( \catGamma \) with codomain~\( \vgamma \),
is a projectively cofibrant resolution of the point. Thus we have
\begin{equation}\label{eq:classical_bousfield_kan}
	{\textstyle\invholim[\catGamma]{ \mF }}
	=
	\End[\vgamma\in\catGamma]{
		\fibrep[output=simplicial]{\mF{\vgamma}}[
			powering={\nerve{{ \catGamma[mod=\vgamma] }}},
		]
	},
\end{equation}
which is the classical form of the Bousfield--Kan formula.

%\Cref{res:bousfield_kan}
%is an immediate corollary to a more general
%result:
%
%\begin{proposition}
%	Let the assumptions be as in
%	\cref{res:bousfield_kan}.
%	Then the functor
%	\[
%		\catsset[diag=\catGamma,proj,spar=\big,op]
%		\longto
%		\catC[diag=\catGamma,proj,spar=\big,smash,diag=\catGamma[op],inj],
%		\qquad
%		\ssetK[arg=\slot]
%		\longmapsto
%		\fibrep{ \mF{slot} }[powering={\ssetK[arg=\slot]}]
%		,
%	\]
%\end{proposition}

The proof relies on the following standard lemma:

\begin{lemmabreak}[{{Lemma \parencite[Proposition~3.6.8]{hovey}}}]\label{res:hovey_lemma}
	Let~\( \catC \) be a model category
	and \( \mF \colon \catsset \to \catC \) a functor preserving colimits and cofibrations. Then \( \mF \)~preserves trivial cofibrations if and only if \( \smash{ \mF{\simp{n}} \to \mF{\simp{0}} } \)~is a weak equivalence for all~\( n \).
\end{lemmabreak}

\begin{proof}[Proof of~\cref{res:bousfield_kan}]
	Clearly,
	\( \fibrep{ \mF{slot} }[ibullet] \)~is a fibrant replacement
	of~\( \mF \)
	with respect to the model structure~\(
		\catC[
			diag=\catdelta[op],
			reedy,
			spar,
			diag=\catGamma,
			proj,
			smash,
		]
	\).
	The theorem will follow if we prove that
	\( \fibrep{ \mF{slot} }[powering={\ssetK[arg=\slot]}] \)~is
	a fibrant replacements of~\( \mF \)
	in~\(
		\catC[
			diag=\catGamma,
			proj,
			spar,
			diag=\catGamma[op],
			inj,
			smash,
		]
	\).
	This will follow from Ken Brown's Lemma if we prove
	that the continuous functor
	\[
		\catsset[diag=\catGamma,proj,spar=\big,op]
		\longto
		\catC[diag=\catGamma,proj,spar=\big,smash,diag=\catGamma[op],inj],
		\qquad
		\ssetK[arg=\slot]
		\longmapsto
		\fibrep{ \mF{slot} }[powering={\ssetK[arg=\slot]}]
		,
	\]
	takes opposites of (trivial) cofibrations to (trivial) fibrations.
	(Trivial) cofibrations
	in~\( \catsset[diag=\catGamma,proj,smash] \)
	are generated from \emph{simple} (trivial) projective
	cofibrations via pushouts and retracts, c.f.~\cref{res:existence_of_proj_inj_model_structures}. Thus by continuity of the functor,
	it suffices to prove the statement for \emph{simple} (trivial) cofibrations.
	We therefore let \(
		\Coprod[\catGamma{\vgamma[0],slot}]{ \ssetK }
		\into
		\Coprod[\catGamma{\vgamma[0],slot}]{ \ssetL }
	\)~be one such, where \( \ssetK\into\ssetL \)
	is a (trivial) cofibration and~\( \vgamma[0] \in \catGamma \).
	This is mapped to
	\[\textstyle
		\Prod[\catGamma{\vgamma[0],slot}]{
			\fibrep{\mF{slot}}[powering=\ssetL]
		}
		\to
		\Prod[\catGamma{\vgamma[0],slot}]{
			\fibrep{\mF{slot}}[powering=\ssetK]
		}.
	\]
	Thus we must show that the composition
	\[
		\catsset[op]
		\xrightarrow{\Prod[\catGamma{\vgamma[0],slot}]}
		\catsset[diag=\catGamma[op],proj,spar=\big,op]
		\longto
		\catC[diag=\catGamma,proj,spar=\big,smash,diag=\catGamma[op],inj],
		\qquad
		\textstyle
		\ssetK
		\longmapsto
		\Prod[\catGamma{\vgamma[0],slot}]{
			\fibrep{ \mF{slot} }[powering={\ssetK}]
		},
	\]
	takes (trivial) cofibrations to (trivial) fibrations.
	Checking that it takes cofibrations to fibrations amounts to checking this for the generating cofibrations~\( \simp{n}[boundary]\into\simp{n} \) in~\( \catsset \). This holds by the assumption that \( \fibrep{\mF{slot}}[ibullet] \) is componentwise Reedy-fibrant.
	Since the functor takes colimits to limits,
	the claim now follows from the (dual of) the lemma.
%%	
%	The simple cofibrations have the form~\(
%		\Coprod[\catGamma{slot,\vgamma[0]}]{ \simp{n}[boundary] }
%		\into
%		\Coprod[\catGamma{slot,\vgamma[0]}]{ \simp{n} }
%	\)
%	for~\( \vgamma[0]\in\catGamma \). These are mapped to
%	\[\textstyle
%		\Prod[\catGamma{slot,\vgamma[0]}]{
%			\fibrep{\mF{slot}}[powering=\simp{n}]
%		}
%		\to
%		\Prod[\catGamma{slot,\vgamma[0]}]{
%			\fibrep{\mF{slot}}[powering=\simp{n}[boundary]]
%		}
%	,
%	\]
%	which are simple fibrations
%	in~\(
%		\catC[diag=\catGamma,proj,spar,diag=\catGamma[op],inj,smash]
%	\).
%	
%	The simple trivial cofibrations
%	have the form~\(
%		\Coprod[\catGamma{slot,\vgamma[0]}]{ \horn[i]{n} }
%		\into
%		\Coprod[\catGamma{slot,\vgamma[0]}]{ \simp{n} }
%	\),
%	where \( \horn[i]{n}\into\simp{n} \)~is the inclusion of the \( i \)th horn.
\end{proof}

\chapter{Homotopy-initial functors}

A functor~\( \mf \colon \catGamma \to \catGamma[prime] \)
is called \textdef{homotopy-initial} if for all objects~\( \vgamma[prime] \in \catGamma[prime] \),
the nerve~\( \nerve{{ \mf[over=\vgamma[prime]] }} \)
is contractible as a simplicial set;
here \( \mf[over=\vgamma[prime]] \) denotes the comma category whose objects are pairs~\( \tup{ \vgamma , \malpha } \)
where \( \malpha \) is a
map~\( \mf{ \vgamma } \to \vgamma[prime] \).
A morphism~\(
	\tup{ {\vgamma[1]} , \malpha }
	\to
	\tup{ {\vgamma[2]} , \malpha }
\)
is a morphism~\( \vgamma[1] \to \vgamma[2] \)
in~\( \catGamma \) making the diagram
\[\begin{tikzcd}[row sep=0em]
	\mf{\vgamma[1]}
	\ar[dd] \ar[rd]
\\
		& \vgamma[prime]
\\
	\mf{\vgamma[2]} \ar[ru]
\end{tikzcd}\]
commute. We aim to reprove the statement
\begin{theorembreak}[{Theorem \parencite[Theorem~19.6.7]{hir}}]\label{res:homotopy_initial}
	Suppose \( \catC \)~is a combinatorial model category
	and~\( \catGamma, \catGamma[prime] \) two categories.
	If \( \mf \colon \catGamma \to \catGamma[prime] \)~is
	homotopy-initial, then we have
	\[\textstyle
		\invholim[\catGamma[prime]]{\mF}
		=
		\invholim[\catGamma]{{\mf[*res]{\mF}}}
	\]
	for all~\( \mF \in \catC[diag=\catGamma[prime]] \).
\end{theorembreak}

This relies on a few technical lemmas:

\begin{lemma}\label{res:kan_extension_nerve}
	If \( \mf \colon \catGamma \to \catGamma[prime] \)
	is a functor, then~\(
		\mf[!kan]{ \nerve{{ \catGamma[over=\slot] }}}
		=
		\nerve{{\mf[over=\slot]}}
		\in
		\catsset[diag=\catGamma[prime]]
	\).
	In particular,
	since \(
		\nerve{{ \catGamma[over=\slot] }}
		\in
		\catsset[diag=\catGamma,proj,smash]
	\)
	is cofibrant,
	\(
		\nerve{{ \mf[over=\slot] }}
		\in
		\catsset[diag=\catGamma[prime],proj,smash]
	\)
	is cofibrant
	by \cref{res:Kan_extensions_Quillen_adjunctions}.
\end{lemma}

\begin{proof}
	Since colimits in diagram categories
	over cocomplete categories can be checked componentwise,
	this boils down to the observation
	\[
		\dirlim[ \mf{\vgamma} \to \vgamma[prime] ]{
			\nerve{{ \catGamma[over=\vgamma] }}{n}
		}
		=
		\nerve{{ \mf[over=\vgamma[prime]] }}{n}
		.
	\]
\end{proof}

The following lemma is inspired
by \textcite[Proposition~19.6.6]{hir}.
See also \textcite[Lemma~8.1.4]{riehl}.

\begin{lemmabreak}\label{res:change_of_diagrams}
	Suppose that \( \catC \)
	is a complete category and that \( \catGamma \)
	and~\( \catGamma[prime] \) are two categories
	with a functor~\( \mf \colon \catGamma \to \catGamma[prime] \).
	Then
	we have
	\[
		\End[\vgamma\in\catGamma]{
			\mF{\mf{\vgamma}}[powering=\nerve{{\catGamma[over=\vgamma]}}]
		}
		=
		\End[\vgamma[prime]\in\catGamma[prime]]{
			\mF{\vgamma[prime]}[powering={
				\nerve{{ \mf[over=\vgamma[prime]] }}
			}]
		}
	\]
	for~\( \mF \in \catC[diag=\catdelta[op],spar,smash,diag=\catGamma] \)
	(see the previous chapter for an explanation of the
	power notation).
\end{lemmabreak}

\begin{proof}
	For the purpose of the proof, we recall
	that the Kan extension formulas
	in~\eqref{eq:kan_extensions} may
	be equivalently written in terms of (co)ends:
	\[
		\mf[!kan]{\mF}{\vgamma[prime]}
		=
		\Coend[\vgamma\in\catGamma]{
		\catGamma[prime]{\mf{\vgamma},\vgamma[prime]}
		\times
		\mF{\vgamma}
		}
		\qquad\text{and}\qquad
		\mf[*kan]{\mF}{\vgamma[prime]}
		=
		\End[\vgamma\in\catGamma]{
			\mF{\vgamma}[powering={\catGamma[prime]{\vgamma[prime],\mf{\vgamma}}}]
		}
		.
	\]
	Here we are using the natural copowering and powering of~\( \catset \) on~\( \catC \),
	given by~\( S \times \vc = \Coprod[S]{\vc} \)
	and~\( \vc[powering=S] = \Prod[S]{\vc} \)
	for~\( S \in \catset \) and~\( \vc \in \catC \),
	which make sense whenever \( \catC \)
	is complete resp.~cocomplete.
	We shall furthermore make use of the
	so-called
	\enquote{co-Yoneda lemma} which says that
	\[
		\mG{\mf{\vgamma}}
		=
		\End[\vgamma[prime]\in\catGamma[prime]]{
			\mG{\vgamma[prime]}[
				powering=\catGamma[prime]{\mf{\vgamma},\vgamma[prime]}
			]
		}
		\qquad
		\text{for all }
		\mG \in \catC[diag=\catGamma[prime]]
		.
	\]
	Finally, we use \enquote{Fubini's theorem} for ends,
	which says that ends, being limits, commute.
	This together yields
	\begin{align*}
		\End[\vgamma\in\catGamma]{
			\mF{\mf{\vgamma}}[powering=\nerve{{\catGamma[over=\vgamma]}}]
		}
		&=
		\End[\vgamma\in\catGamma]{
			\End[\ordset{n}\in\catdelta]{
				\mF{\mf{\vgamma}}[n,powering=\nerve{{\catGamma[over=\vgamma]}}{n}]
			}
		}
	\\
		&=
		\End[\vgamma\in\catGamma]{
			\End[\ordset{n}\in\catdelta]{
				\End[\vgamma[prime]\in\catGamma[prime]]{
					\mF{\vgamma[prime]}[
						n,
						powering=\catGamma[prime]{\mf{\vgamma},\vgamma[prime]},
						spar=\big,
						powering=\nerve{{\catGamma[over=\vgamma]}}{n},
					]
				}
			}
		}
\\
		&=
		\End[\vgamma\in\catGamma]{
			\End[\ordset{n}\in\catdelta]{
				\End[\vgamma[prime]\in\catGamma[prime]]{
					\mF{\vgamma[prime]}[
						n,
						powering={
							\catGamma[prime]{\mf{\vgamma},\vgamma[prime]}
							\times
							\nerve{{\catGamma[over=\vgamma]}}{n}
						},
					]
				}
			}
		}
\\
		&=
		\End[\ordset{n}\in\catdelta]{
			\End[\vgamma[prime]\in\catGamma[prime]]{
				\mF{\vgamma[prime]}[n,powering={
					\Coend[\vgamma\in\catGamma]{
						\catGamma[prime]{\mf{\vgamma},\vgamma[prime]}
						\times
						\nerve{{\catGamma[over=\vgamma]}}{n}
					}
				}]
			}
		}
\\
		&=
			\End[\vgamma[prime]\in\catGamma[prime]]{
				\mF{\vgamma[prime]}[n,powering={
					\mf[!kan]{{
						\nerve{{ \catGamma[over=\slot] }}[arg=\vgamma[prime]]
					}}
				}]
			}
		=
		\End[\vgamma[prime]\in\catGamma[prime]]{
			\mF{\vgamma[prime]}[powering={
				\nerve{{ \mf[over=\vgamma[prime]] }}
			}]
		}
	\end{align*}
	where the last equality sign is due
	to \cref{res:kan_extension_nerve}.
\end{proof}

\begin{proof}[Proof of~\cref{res:homotopy_initial}]
	\Cref{res:bousfield_kan} and equation~\eqref{eq:classical_bousfield_kan}
	show that
	\[
		{
			\textstyle
			\invholim[\catGamma]{ \mf[*res]{\mF} }
		}
		=
		\End[\vgamma[prime]\in\catGamma[prime]]{
			\fibrep{\mF{\vgamma[prime]}}[powering={
				\nerve{{ \mf[over=\vgamma[prime]] }}
			}]			
		}
		.
	\]
	Since \( \nerve{{ \mf[over=\vgamma[prime]] }} \)
	is contractible for all~\( \vgamma[prime]\),
	\( \nerve{{ \mf[over=\slot ] }} \)~is a projectively cofibrant
	resolution of the point
	by \cref{res:kan_extension_nerve}.
	Thus the \anrhs is
	exactly~\( \invholim[\catGamma[prime]]{\mF}\)
	by~\cref{res:bousfield_kan}.
%	\cref{res:hovey_lemma} shows
%	that we have a weak equivalence~\(
%		\fibrep{\mF{\vgamma[prime]}}[powering={
%			\nerve{{ \mf[over=\vgamma[prime]] }}
%		}]
%		\cong
%		\fibrep{\mF{\vgamma[prime]}}[powering={
%			*
%		}]
%		=
%		\fibrep{\mF{\vgamma[prime]}}[0]
%		\cong
%		\mF{\vgamma[prime]}
%	\).
%	Thus the \anrhs of the above equation is
\end{proof}

\begin{examplebreak}[Example: Fat totalization formula]\label{ex:fat_totalization}
	Recall from \cref{ex:delta_reedy} that the simplex category~\( \catdelta \)
	is Reedy with~\( \catdelta[plus] \)
	being the subcategory containing only injective maps.
	The inclusion \(\smash{ \miota\colon \catdelta[plus]\into\catdelta }\) is
	homotopy-initial
	(\cite[see e.g.][Example~8.5.12]{riehl}
	or \cite[Example 21.2]{dugger}), hence
	\(\smash{
		\invholim[\catdelta]{ \cosimplicial{X}{*} } 
		=
		\invholim[\catdelta[plus]]{ \cosimplicial{X}{*} }
	}\)
	for all~\( \smash{ \cosimplicial{X}{*} \in \catC[diag=\catdelta] } \).
	As \( \catdelta[plus] \)~is a direct category, we obtain
	from~\cref{res:holim_direct_category} that we may
	calculate~\( \invholim[\catdelta]{ \cosimplicial{X}{*} } \)
	as
	\[\textstyle
		\invholim[\catdelta] { \cosimplicial{X}{*} }
		=
		\End[\catdelta[plus]]{
			\fibrep[output=simplicial] (
				\cosimplicial{X}{n}
			)_{n}
		}
	\]
	for some functor
	\( \smash{
		\fibrep
		\colon
		\catC\to\catC[diag={\catdelta[plus,op]}]
	} \)
	that takes~\( \vx \) to an injectively (i.e.~Reedy-) fibrant replacement of the constant diagram at~\( \vx \).
%	This is the same technique used in \textcite{hir} when he develops
%	homotopy limits via “simplicial frames” on diagrams.
	This is the so-called \textdef{fat totalization}
	formula for homotopy limits over~\( \catdelta \).
	The dual formula for homotopy colimits over~\( \catdelta[op] \)
	is called the \textdef{fat geometric realization}
	formula.
\end{examplebreak}

\section*{Acknowledgements}

W would like to thank Edouard Balzin, Marcel B\"okstedt, and Stefan Schwede for many fruitful discussions and for reading through a draft of this paper. Special thanks to Henning Haahr Andersen for many years of great discussions, help, and advice, and for making our cooperation possible in the first place. This paper was written mostly while the authors were visiting the Max Planck Institute for Mathematics in Bonn, Germany. We would like to express our gratitude to the institute for inviting us and for providing us with an excellent and stimulating working environment.

\endgroup

\pagebreak

%% file: ms.bib
@book{htt,
	hyphenation={american},
	author={Lurie, Jacob},
	title={Higher Topos Theory},
	year={2009},
	series={Annals of Mathematics Studies},
	number={170},
	publisher={Princeton University Press},
}

@article{gambino_simplicial,
	author = {Gambino, Nicola},
	title = {Weighted limits in simplicial homotopy theory},
	journal = {Journal of Pure and Applied Algebra},
	volume = {214},
	number = {7},
	pages = {1193 - 1199},
	year = {2010},
	%issn = "0022-4049",
	%doi = "https://doi.org/10.1016/j.jpaa.2009.10.006",
	%url = "http://www.sciencedirect.com/science/article/pii/S0022404909002412",
}

@book{hir,
	author={Hirschhorn, Philip S.},
	year={2003},
	title={Model Categories and Their Localizations},
	publisher={American Mathematical Society},
	series={Mathematical Surveys and Monographs},
	volume={99},
}

@online{dugger,
	year={2008},
	hyphenation={american},
	author={Dugger, Daniel},
	title={A primer on homotopy colimits},
	url={http://pages.uoregon.edu/ddugger/hocolim.pdf},
}

@book{hovey,
	hyphenation={american},
	author={Hovey, Mark},
	title={Model categories},
	year={1999},
	publisher={American Mathematical Society},
	series={Mathematical Surveys and Monographs},
	volume={63},
	url={https://web.math.rochester.edu/people/faculty/doug/otherpapers/hovey-model-cats.pdf},
}

@book{riehl,
	hyphenation={american},
	author={Riehl, Emily},
	year={2014},
	title={Categorical Homotopy Theory},
	publisher={Cambridge University Press},
}

@article{explicit,
	author={Block, Jonathan and Holstein, Julian {V.S.} and Wei, Zhaoting},
	title={Explicit homotopy limits of dg-categories and twisted complexes},
	year={2017},
	journal={Homology, Homotopy and Applications},
	volume={19(2)},
	pages={343-371},
}

@article{comodules,
	hyphenation={american},
	author={Arkhipov, Sergey and Ørsted, Sebastian},
	title = {Homotopy limits in the category of dg-categories in terms of $\mathrm{A}_{\infty}$-comodules},
	year={2018},
	eprinttype={arxiv},
	eprint={1812.03583},
}
